\documentclass[12pt,reqno]{amsart}

\usepackage{a4,amsmath,amssymb}
\usepackage{enumitem}
\usepackage{tikz}
\usetikzlibrary{arrows.meta}
\usepackage{color}

\theoremstyle{plain}
\newtheorem{thm}{Theorem}[section]

\newtheorem{lemma}[thm]{Lemma}

\newtheorem{cor}[thm]{Corollary}

\theoremstyle{definition}

\theoremstyle{remark}
\newtheorem{exa}[thm]{Example}
\newtheorem{que}[thm]{Question}

\newcommand{\N}{\mathbb{N}}

\DeclareMathOperator{\Z}{\mathbb{Z}}

\newcommand{\Id}[1]{\mathrm{Id}\langle {#1}\rangle}
\newcommand{\Sub}[1]{\langle {#1}\rangle}
\renewcommand{\sp}[1]{\mathrm{Span}({#1})}
\newcommand{\K}[1]{K\langle{#1}\rangle}
\newcommand{\ZZ}[1]{\mathbb{Z}\langle{#1}\rangle}

\newcommand{\G}[2]{\langle {#1}\rangle_{#2}\, }
\renewcommand{\t}[3]{[{#1},{#2},{#3}]}
\newcommand{\vg}[1]{[{#1}]}

\newcommand{\setsuchthat}{\: :\:}
\newcommand{\id}{\mathrm{id}}

\date{\today}
\thanks{The first author was supported by the National Science Foundation under Grant No. DMS 1500254.}
%; Secondary 16A06}

\begin{document}

%%%%%%%%%%%%%%%%%%%%%%%%%%%%%%%%%%%%%%%%%%%%%%%%%%%%%%%%%%%%%%%%%%%
%%  FRONT MATTER                                                 %%
%%%%%%%%%%%%%%%%%%%%%%%%%%%%%%%%%%%%%%%%%%%%%%%%%%%%%%%%%%%%%%%%%%%

\title{Presentations for subrings and subalgebras of finite co-rank}

\author{Peter Mayr} 
\address[Peter Mayr]{Department of Mathematics, CU Boulder, USA}
\email{peter.mayr@colorado.edu}
\author{Nik Ru\v{s}kuc}
\address[Nik Ru{\v{s}}kuc]{School of Mathematics and Statistics, University of St Andrews, St Andrews, Scotland, UK}
\email{nik.ruskuc@st-andrews.ac.uk}

\begin{abstract}
Let $K$ be a commutative  Noetherian ring with identity, let $A$ be a $K$-algebra, and let $B$ be a subalgebra of $A$
 such that $A/B$ is finitely generated as a $K$-module.
The main result of the paper is that $A$ is finitely presented (resp. finitely generated) if and only if $B$ is finitely presented (resp. finitely generated).
As corollaries we obtain: a subring of finite index in a finitely presented ring is finitely presented; 
a subalgebra of finite co-dimension in a finitely presented algebra over a field
is finitely presented (already shown by Voden in 2009). 
We also discuss the role of the Noetherian assumption on $K$, and show that
for finite generation it can be replaced by a weaker condition that the module $A/B$ be finitely presented.
Finally, we demonstrate that the results do not readily extend to non-associative algebras, by exhibiting an ideal of co-dimension $1$ of the free Lie algebra of rank 2 which is not finitely generated as a Lie algebra.
\end{abstract}

\maketitle

\section{The results}
\label{sec-intro}

Throughout this paper, $K$ will be a commutative ring with $1$.
A \emph{$K$-algebra} is a structure $A$ which is simultaneously a ring (not necessarily commutative or with $1$) and a $K$-module, such that the ring multiplication is $K$-bilinear.
Thus, every ring is a $\mathbb{Z}$-algebra, and the classical algebras are the special case where $K$ is a field. All rings and algebras throughout will be associative.

The main result proved in this paper is the following:

\begin{thm}
\label{thm-i1}
 Let $A$ be a $K$-algebra (not necessarily commutative or with $1$) over a commutative Noetherian ring
 $K$ with $1$, and let $B$ be a subalgebra of $A$ such that $A/B$ is a finitely generated $K$-module.

 Then $A$ is finitely presented as a $K$-algebra  if and only if $B$ is finitely presented as a $K$-algebra.
\end{thm}

Specialising to rings and algebras over fields we have the following immediate corollaries:

\begin{cor}
\label{cor-i2}
Let $R$ be a ring (not necessarily commutative or with $1$), and let $S$ be a subring such that $R/S$ is a finitely generated abelian group.

Then $R$ is finitely presented if and only if $S$ is finitely presented.
\end{cor}

\begin{cor}
\label{cor-i3}
Let $R$ be a ring (not necessarily commutative or with $1$), and let $S$ be a subring of finite index.

Then $R$ is finitely presented if and only if $S$ is finitely presented.
\end{cor}

\begin{cor}[{\cite[Corollary 5]{voden09}, \cite[Section VII]{lewin69}}]
\label{cor-i4}
Let $A$ be an algebra over a field $K$, and let $B$ be a subalgebra of finite co-dimension.

Then $A$ is finitely presented if and only if $B$ is finitely presented.
\end{cor}

A necessary initial step in proving Theorem \ref{thm-i1} is to establish a generating set for $B$, given a generating set for $A$, which establishes a finite generation analogue of Theorem \ref{thm-i1}. 
In fact, we can prove a stronger result, without the Noetherian assumption
on $K$ (recalling that for modules over Noetherian rings, finite generation and finite presentability are equivalent properties):

\begin{thm}
\label{cor-i5}
Let $A$ be a $K$-algebra (not necessarily commutative or with $1$) over a commutative ring
 $K$ with $1$, and let $B$ be a subalgebra of $A$ such that $A/B$ is a finitely presented $K$-module.

Then $A$ is finitely generated as a $K$-algebra  if and only if $B$ is finitely generated as a $K$-algebra.
\end{thm}

 The finite generation result has been known for rings and for algebras over fields since the 1960s from the work
 of Lewin: the non-obvious direction ($\Rightarrow$) for subrings of finite index is the main result of
 \cite{lewin67}, while for subalgebras of finite co-dimension it is implicitly present in \cite[Section VII]{lewin69}.

 Let us point out a key difference between our situation of general $K$-algebras and commutative rings.
 For $K$ a commutative Noetherian ring with $1$, the ring of polynomials $K[x_1,\dots,x_n]$ in
 commuting variables $x_1,\dots,x_n$ over $K$ is Noetherian by Hilbert's Basis Theorem.
 Every finitely generated commutative $K$-algebra with $1$ is a homomorphic image of some $K[x_1,\dots,x_n]$
 and consequently finitely presented.
 In particular being finitely generated is equivalent to being finitely presented for commutative rings
 but not in general.

The paper is organised as follows. 
The next section contains the proof of Theorem \ref{thm-i1}, with the proofs of two subsidiary results postponed to Sections \ref{sec-ext} and \ref{sec-ri}. The latter of the two, asserting that a right ideal of finite co-rank in a free $K$-algebra of finite rank is finitely presented, is the key non-trivial part of the argument.
We prove Theorem \ref{cor-i5} in Section~\ref{sec-conca}. The paper concludes with Section~\ref{sec-concb} 
in which we demonstrate that finite presentability of $A/B$ is a necessary condition for finite generation (Theorem~\ref{cor-i5}), and that $K$ being Noetherian is necessary for finite presentation (Theorem~\ref{thm-i1}).
We also discuss parallels with and differences from the Reidemeister--Schreier Theorem from combinatorial group theory, as well as the theory of subsemigroups of free semigroups.
Finally, we briefly discuss non-associative algebras, and show that the analogues of our results fail already for Lie algebras, in that there
exists an ideal of co-dimension $1$ in the free Lie algebra of rank $2$ which is not finitely generated as a Lie algebra.

\section{Proof outline}
\label{sec-outline}

\subsection{Notation and terminology} \label{subs-2.1}

As indicated at the beginning of the paper, $K$ will continue to stand for a fixed commutative ring with $1$.
For a $K$-algebra $A$ and a set $U\subseteq A$ the $K$-subalgebra generated by $U$ will be denoted by $\Sub{U}$, while the ideal generated by $U$ will be denoted by $\Id{U}$.
As our $K$-algebras do not (necessarily) have 1,
at times we will need to make use of a formal identity $1$ not belonging to the $K$-algebra under consideration.
We will then write $A^1$ for the $K$-algebra obtained by adjoining $1$ to $A$.
For a $K$-module $C$ and a set $V\subseteq C$, the submodule generated by $V$ will be denoted by $\sp{V}$.
If $D$ is a submodule of $C$ and the quotient $C/D$ is finitely generated we say that $D$ has a \emph{finite co-rank} in $C$.\footnote{One could define \emph{the} co-rank of $D$ in $C$ as the minimum cardinality of a generating set of $C/D$, but we will not need this notion.} When $K$ is a field, we will use the more customary term \emph{co-dimension}, rather than co-rank.

The \emph{free $K$-algebra} $\K{X}$ is the semigroup ring of the free semigroup $X^+$ with coefficients in $K$.
It consists of all \emph{polynomials} over $X$,
i.e. (finite) $K$-linear combinations of monomials (non-empty words) over $X$.
Note that the variables in $X$ do not commute, and $\K{X}$ does not have a (multiplicative) identity.
The free monoid over $X$, which includes the empty word $1$, will be denoted by $X^\ast:=X^+\cup\{1\}$.
The $K$-algebra $\K{X}^1$ is then isomorphic to the semigroup ring of $X^\ast$ with coefficients from $K$.

A $K$-algebra $A$ is \emph{finitely generated} if there exists a finite set $U$ such that 
$\Sub{U}=A$. Equivalently, $A$ is finitely generated if it is isomorphic to a quotient 
$\K{X}/I$ of a free $K$-algebra with $X$ finite.
If in addition $I$ is finitely generated as an ideal, $A$ is said to be \emph{finitely presented}.
When dealing with ideals, we will use the term `finitely generated' to mean `finitely generated as an ideal'. If another meaning is intended, say as a $K$-algebra, then we will say so explicitly.

\subsection{Reduction to ideals} \label{subs-2.2}

In proving Theorem \ref{thm-i1} we will prefer to work with ideals rather than subalgebras,
which is facilitated by the following result (cf. \cite[Lemma 1]{lewin67}):

\begin{lemma}
\label{le-o1}
Suppose $K$ is Noetherian, and let $A$ be a $K$-algebra with subalgebra $B$ of finite co-rank.
 Then $B$ contains an ideal $I$ of $A$ with finite co-rank in $B$
(and hence in $A$ as well). 
\end{lemma}

\begin{proof}
 By the assumptions, we have a finite set $V \subseteq A$ such that $B+\sp{V} = A$.
Let $V_1:=V\cup\{1\}$, and define a $K$-module homomorphism
\[ h\colon B \to (A/B)^{V_1\times V_1} \]
 by 
\[ h(x)_{u,v} := uxv+B \ (x\in B,\ u,v\in V_1), \]
where $h(x)_{u,v}$ denotes the $(u,v)$-component of the tuple $h(x)$.
 Let $I$ be the ideal of $A$ generated by the kernel $H$ of $h$. 
Note that $uHv\subseteq B$ for all $u,v\in V_1$.
Letting $S:=\sp{V_1}$, a $K$-submodule of $A^1$,
it follows that $SHS,SH,HS\subseteq B$. Now
\begin{align*}
A^1HA^1  &=(B+S)H(B+S)=BHB+B(HS)+(SH)B+SHS\\
&\subseteq BBB+BB+BB+B=B,
\end{align*}
and hence $I\subseteq B$.
Since $K$ is Noetherian and $A/B$ is a finitely generated $K$-module, the image of $h$ is also finitely generated.
Hence, by the 1st Isomorphism Theorem, $I$ has finite co-rank in $B$.
\end{proof}

It now readily follows that it is sufficient to prove Theorem \ref{thm-i1}
 in the case where $B$ is an ideal.
Indeed, let $B$ be a $K$-subalgebra of finite co-rank, and  let $I\subseteq B$ be an ideal of $A$ of finite co-rank.
Since $K$ is assumed to be Noetherian, $I$ has finite co-rank in $B$ as well, so two applications of the Theorem \ref{thm-i1}  for ideals yield: $A$ is finitely presented if and only if $I$ is finitely presented if and only if $B$ is finitely presented.

\subsection{Proof for ideals} \label{subs-2.3}

 It is easy and perhaps well known that if an ideal $B$ of finite co-rank in $A$ is a finitely generated
 (respectively, finitely presented) algebra, then so is $A$. 
 We will in fact prove, in Section \ref{sec-ext}, a more general (and equally elementary) result:

\begin{lemma}
\label{le-o2}
Let $A$ be a $K$-algebra, and let $B$ be an ideal of $A$.
If both $B$ and $A/B$ are finitely generated (respectively, finitely presented) as $K$-algebras, then
$A$ itself is finitely generated (respectively, finitely presented).
\end{lemma}

Note that if $A/B$ is a finitely presented $K$-module, then it is a finitely presented $K$-algebra.
Indeed, a finite presentation for $A/B$ as a $K$-algebra is then obtained from its $K$-module presentation by adding relations which express the product of any two generators as a $K$-linear combination of generators.
In the case when $K$ is Noetherian, if the $K$-module $A/B$ is finitely generated then it is also finitely presented, 
and the backwards direction of Theorem \ref{thm-i1} follows.

 The brunt of the work in proving Theorem \ref{thm-i1}  is contained in establishing the forward direction,
 from $A$ to $B$, particularly for the case where $A$ is a free $K$-algebra: 

\begin{lemma}
\label{le-o3}
If $K$ is Noetherian, then every right ideal $R$ of finite co-rank in a free $K$-algebra $\K{X}$ of finite rank is finitely presented.
\end{lemma}

This will be proved in Section \ref{sec-ri}.
The move to the general case is achieved via the following:

\begin{lemma}[cf. {\cite[Lemma 4]{voden09}}]
\label{le-o4}
Let $R$ be a subalgebra of finite co-rank in a free $K$-algebra $\K{X}$, and let $I\subseteq R$ be a
 finitely generated ideal of $\K{X}$. Then $I$ is also finitely generated as an ideal of $R$. 
\end{lemma}

\begin{proof}
Voden proves this for the case of algebras over a field, in \cite[Lemma 4]{voden09}.
 However, the proof does not use the assumption that $K$ be a field, and therefore remains valid in our more
 general setting where $K$ is any commutative ring with identity.
\end{proof}

 With the above lemmas and observations in hand, the proof of Theorem \ref{thm-i1} can be rapidly concluded. 
 Suppose that $A$ is a finitely presented $K$-algebra, and write it as $A=\K{X}/I$, with $X$ a finite set, and
 $I$ a finitely generated ideal of $\K{X}$. 
Let $B$ be an ideal of finite co-rank in $A$; as observed at the end of Subsection \ref{subs-2.2} it suffices to show that $B$ is finitely presented.
 Then $\K{X}$ has an ideal $R$ containing $I$ such that $B = R/I$, and $R$ has finite co-rank in $\K{X}$.
 Hence $R$ is finitely presented by Lemma \ref{le-o3}.
 Moreover $I$ is finitely generated as an ideal of $R$ by Lemma~\ref{le-o4}. Thus $R/I=B$ is finitely presented, as desired. 
 
So, to complete the proof of Theorem~\ref{thm-i1} 
 it remains to prove Lemmas \ref{le-o2}
 and \ref{le-o3}, which we proceed to do in the next two sections.

\section{Presentations for extensions (Lemma~\ref{le-o2})}
\label{sec-ext}

This section constitutes the proof of Lemma \ref{le-o2}: 
{\it For a $K$-algebra $A$ and an ideal $B$, if both $B$ and $A/B$ are finitely generated (resp. finitely presented) as $K$-algebras, then
$A$ itself is finitely generated (resp. finitely presented).}

The finite generation part is easy: if $A/B=\Sub{\{ x+B\setsuchthat x\in X\}}$ and
 $B=\Sub{Y}$, then $A=\langle X\cup Y\rangle$.

 For finite presentability suppose  
\[  
A= \K{X}/I,\ B = R/I 
 \]
 for a finite set $X$ and ideals $I\subseteq R$ of the free $K$-algebra $\K{X}$. Further suppose that $\K{X}/R$,
 which is isomorphic to $A/B$, and $R/I$ are finitely presented as $K$-algebras. We will show that
 $I$  is a finitely generated ideal of $\K{X}$.

 Since $R$ is finitely generated as an ideal of $\K{X}$, and since $R/I$ is a finitely generated $K$-algebra, 
 there is a finite set $Y$ and a homomorphism $\K{Y}\to \K{X}$, $t\mapsto \overline{t}$, such that 
\[ R = \Id{\overline{Y}} \text{ and } R = \Sub{ \overline{Y}} + I. \]
 Note that
\[ 
B = \frac{\Sub{ \overline{Y}} + I}{I} \cong \frac{\Sub{ \overline{Y}}}{\Sub{\overline{Y}}\cap I}. 
\]
Now consider the epimorphism
\[ 
\psi\colon \K{Y}\to \frac{\Sub{\overline{Y}}}{\Sub{ \overline{Y}}\cap I},\]
 defined by 
\[ \psi(y) := \overline{y}+\Sub{ \overline{Y}}\cap I \ \ (y\in Y). \]
 Since $B$ is finitely presented, $\ker \psi = \Id{W}$ for a finite $W\subseteq \K{Y}$ with $\overline{W}\subseteq\Sub{ \overline{Y}}\cap I$.
 
 Since $R = \Sub{ \overline{Y}} + I$ is an ideal of $\K{X}$, for all $x\in X$, $y\in Y$ there exist
 $p_{x,y},p_{y,x}\in \K{Y}$ such that the polynomials $x\overline{y}-\overline{p}_{x,y}$ and $\overline{y}x-\overline{p}_{y,x}$ both belong to $I$. 
We claim that
\begin{equation}
\label{eq-IJ}
 I = \Id{\overline{W}\cup\{ x\overline{y}-\overline{p}_{x,y}, \overline{y}x-\overline{p}_{y,x} \setsuchthat x\in X, y\in Y \} }.
\end{equation}
The inclusion ($\supseteq$) is obvious from the preceding discussion.
 For ($\subseteq$),
denote by $J$ the ideal on the right-hand side of \eqref{eq-IJ}, and let $u\in I$.
Then $u\in R$, which is equal to $\Id{\overline{Y}}$, and
 so there exist $n\in\N$, $y_i\in Y$, $r_i,s_i\in \K{X}^1$ ($i=1,\dots,n$) such that  
\[ u = r_1\overline{y}_1s_1+\dots+r_n\overline{y}_ns_n. \]
On the other hand, using the polynomials
$x\overline{y}-\overline{p}_{x,y}, \overline{y}x-\overline{p}_{y,x}\in J$,
we see that 
\[ \forall w\in \K{Y}\ \forall a,b\in \K{X}^1\ \exists v\in \K{Y}\colon a\overline{w}b - \overline{v} \in J; \]
this can be formally proved by a straightforward
induction on the complexity of polynomials $a$ and $b$. 
 In particular there exists $v\in \K{Y}$ such that $u-\overline{v}\in J$.
 But then $\overline{v}\in I$, which implies $v\in\ker \psi$ and $\overline{v}\in\Id{\overline{W}}$.
 Thus $\overline{v}$ is in
 $J$, and so is $u$. This proves that $I=J$, from which it follows that $I$ is indeed finitely generated.
 Hence $A=\K{X}/I$ is finitely presented and Lemma~\ref{le-o2} is proved.

\section{Right ideals in free $K$-algebras (Lemma \ref{le-o3})}
\label{sec-ri}

This section contains a proof of Lemma \ref{le-o3}:
{\it for a Noetherian $K$, a right ideal $R$ of finite co-rank in a  finitely generated free $K$-algebra $\K{X}$ is finitely presented.}

\subsection{Generators} \label{sec-gens}

First we need to identify a suitable (finite) set of generators for $R$.
In doing this we broadly follow Lewin \cite{lewin67}. However, our context is more general ($K$-subalgebras of finite co-rank vs. subrings of finite index), and also we need to set up supporting infrastructure of free $K$-algebras and homomorphisms in preparation for the defining relations part of the proof.
A diagram illustrating the mutual relationships between these objects is presented in Figure \ref{fig:1}, which the reader may find helpful to consult from time to time as they go through the technical argument of this section.

\begin{figure}
\begin{center}

\begin{tikzpicture}

\node (KZ) at (0,0) {$\K{Z}$};
\node (KTV) at (2,1) {$\K{T}+\sp{V}$};
\node (KTpiU) at (2,2) {$\K{T}+\sp{\pi(U)}$};

\node (KTU) at (2,4) {$\K{T}+\sp{U}$};
\node (KY) at (0,5) {$\K{Y}$};

\node (KX) at (9,1) {$\K{X}^1$};
\node (R) at (9,4) {$R$};

\begin{scope}[arrows={-{Stealth}[length=3mm,width=2mm]}]
\draw (KY)--(KZ);
\draw ([xshift=-1.5mm]KTU.south)--([xshift=-1.5mm]KTpiU.north);
\draw ([xshift=1.5mm]KTpiU.north)--([xshift=1.5mm]KTU.south);
\draw (KZ)--(KX);
\draw (KX)--(KTV);
\draw (KY)--([yshift=1mm]R.west);
\draw (R)--(KTU);
\draw (R)--(KTpiU);
\end{scope}

\node [rotate=90] at (2,1.5) {$\supseteq$};
\node [rotate=45] at (0.8,0.5) {$\supseteq$};
\node [rotate=315] at (0.6,4.55) {$\supseteq$};
\node [rotate=90,scale=1.5] at (9,2.5) {$\supseteq$};

\node at (5.3,3.7) {$\overline{\varphi}=\rho\varphi$};
\node at (6.3,2.8) {$\varphi$};
\node at (6.3,1.3) {$\varphi$};
\node at (4.2,0.2) {$\psi$};
\node at (4.2,5) {$\overline{\psi}=\psi\pi$};
\node at (-0.3,2.5) {$\pi$};
\node at (1.6,3.1) {$\pi$};
\node at (2.4,3.0) {$\rho$};

\end{tikzpicture}

\end{center}
\caption{Proof of Lemma \ref{le-o3}: algebras, modules and homomorphisms.}
\label{fig:1}
\end{figure}
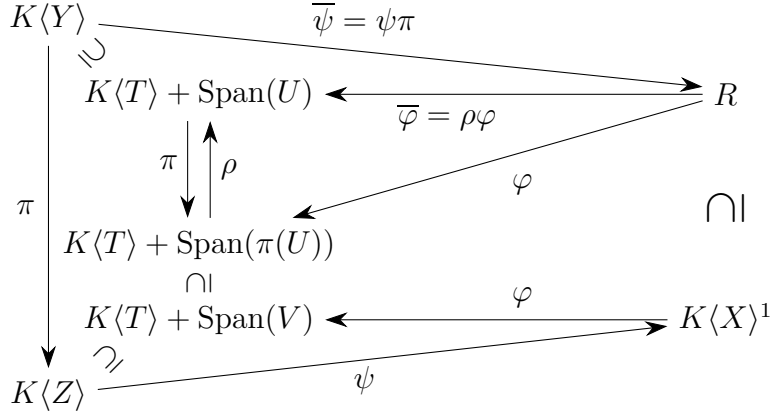

 By the finite co-rank assumption we have a finite subset $B$ of $\K{X}$ such that $\K{X} = R+\sp{B}$.
 Let $\langle . \rangle$ denote a map $\K{X}\rightarrow K^B,\ p\mapsto \G{p}{},$ that is constant on all cosets
 of $R$ in $\K{X}$, and satisfies
\begin{align*}
p-\sum_{b\in B} \G{p}{b} b \in R &\text{ for all } p\in \K{X},\\
\G{0}{b}=0 &\text{ for all } b\in B.
\end{align*}
 Here $\G{p}{b}$ denotes the $b$-th component of the tuple $\G{p}{}$. 
An alternative interpretation of $\G{p}{}$ is as follows:
pick an arbitrary collection of unique coset representatives modulo $R$, write each representative as a $K$-linear
 combination of $K$-module generators from $B$, and record the coefficients of the representative of $p$ in the tuple
 $\G{p}{}$.

 We introduce the following sets of new symbols:
\begin{align*}
 T & := \bigl\{\t{a}{x}{b} \setsuchthat x\in X,\ a,b\in B\cup\{1\} \bigr\}, \\
 V & := \bigl\{ \vg{b} \setsuchthat b\in B\cup\{1\} \bigr\}, \\
 Z & := T\cup V,
\end{align*}
 and define a homomorphism $\psi\colon \K{Z} \to \K{X}^1$ by
\[
 \psi\colon \t{a}{x}{b}\mapsto \Bigl(ax-\sum_{c\in B} \G{ax}{c} c\Bigr)b \text{ and } 
 \vg{b}\mapsto b.
\]
 Since $R$ is a right ideal,
\begin{equation} \label{eq:psiT}
 \psi(T) \subseteq R.
\end{equation}

\begin{lemma} \label{le:FX1module}
 The free $K$-module $\K{T} + \sp{V}$ is a right $\K{X}^1$-module under the following action
$(s,x)\mapsto s*x$
 of $X$ on its
 $K$-basis $T^+\cup V$: 
 for $r\in T^*,\t{a}{y}{b}\in T,\vg{b}\in V$, and $x\in X$ let
\begin{align}
\label{eq:*1}
 \bigl(r\t{a}{y}{b}\bigr)*x & := r\Bigl(\t{a}{y}{1}\t{b}{x}{1}+\sum_{c\in B} \G{bx}{c} \t{a}{y}{c}\Bigr), \\
\label{eq:*2}
 \vg{b}*x & := \t{b}{x}{1}+\sum_{c\in B} \G{bx}{c} \vg{c}.
\end{align} 
\end{lemma}

\begin{proof}
The above action of a single $x\in X$ extends uniquely to a $K$-endomorphism on $\K{T} + \sp{V}$ by its freeness as a $K$-module.
 And then this linear action of elements of $X$ on $\K{T} + \sp{V}$ extends uniquely to an action of $\K{X}^1$ by its freeness as a $K$-algebra.
\end{proof}

 Note that 
\[ \varphi\colon \K{X}^1\to\K{T}+\sp{V},\ p\mapsto [1]*p, \]
 is a $\K{X}^1$-module homomorphism.
 We show that $\psi$ restricts to a $\K{X}^1$-module homomorphism from $\K{T} + \sp{V}$ onto $\K{X}^1$
 with right inverse $\varphi$.

\begin{lemma} \label{le:Psirp}
 For all $r\in \K{T} + \sp{V}$ and $p\in \K{X}^1$ we have 
\begin{equation} \label{eq:psirp}
 \psi(r*p) = \psi(r)p. 
\end{equation}
 Further $\psi\varphi = \id_{\K{X}^1}$, the identity map on $\K{X}^1$. 
 \end{lemma}

\begin{proof}
 Since $*$, $\psi$, and the multiplication in $\K{X}^1$ are all $K$-linear, it suffices to prove~\eqref{eq:psirp}
 for monomials $r\in T^+\cup V$ and $p\in X^\ast$. We use induction on the length of $p$. For $p=1$ there is nothing
 to prove. For $p=x$ in $X$ apply the homomorphism $\psi$ to the defining equalities for $*$ in
 Lemma~\ref{le:FX1module}. Starting with the shorter second one, let $\vg{b}\in V$. Then
\begin{align*}
 \psi\bigl(\vg{b}*x\bigr) & = \psi\Bigl(\t{b}{x}{1}+\sum_{c\in B} \G{bx}{c} \vg{c}\Bigr) \\
 & = \Bigl(bx-\sum_{c\in B} \G{bx}{c} c\Bigr)1 + \sum_{c\in B} \G{bx}{c} c \\
 & = bx \\
 & = \psi\bigl(\vg{b}\bigr) x.
\end{align*}
 Similarly, for \eqref{eq:*1} let $r\in T^*$, $\t{a}{y}{b}\in T$. Then
\begin{align*}
 & \psi\bigl((r\t{a}{y}{b})*x\bigr) \\
= & \psi(r)\, \psi\Bigl(\t{a}{y}{1}\t{b}{x}{1}+\sum_{c\in B} \G{bx}{c} \t{a}{y}{c}\Bigr) \\
= & \psi(r)\, \Bigl( \bigl(ay-\sum_{c\in B} \G{ay}{c} c\bigr) \bigl(bx-\sum_{c\in B} \G{bx}{c} c\bigr)+\sum_{c\in B} \G{bx}{c}\bigl(ay-\sum_{d\in B} \G{ay}{d} d\bigr)c \Bigr) \\
= &\psi(r)\, \Bigl( \bigl(ay-\sum_{c\in B} \G{ay}{c} c\bigr)bx - \bigl(ay-\sum_{c\in B} \G{ay}{c} c\bigr)\bigl(\sum_{c\in B} \G{bx}{c} c\bigr)\\
& \hspace{75mm} +\bigl(ay-\sum_{d\in B} \G{ay}{d} d\bigr)\bigl(\sum_{c\in B} \G{bx}{c}c\bigr)\Bigr) \\
= &\psi(r)\, \bigl(ay-\sum_{c\in B} \G{ay}{c} c\bigr)bx \\
= & \psi\bigl(r\t{a}{y}{b}\bigr)x.
\end{align*}
 This concludes the base case of our induction. For the inductive step, let $p = qx$ with $q\in X^+$, $x\in X$, and let $r\in  T^+\cup V$.
 Using module properties and the induction assumption we obtain
\[ \psi(r*(qx)) = \psi((r*q)*x) = \psi(r*q)x = \psi(r)qx, \]
completing the proof of~\eqref{eq:psirp}.
 The seconds assertion follows immediately from~\eqref{eq:psirp} by setting $r=[1]$.
\end{proof}

At first glance, the significance of the second assertion in Lemma \ref{le:Psirp} may be unclear: it effectively
 establishes (the image of) $Z$ as an alternative generating set for the free $K$-algebra $\K{X}^1$. However,
 it achieves two further things at the same time. Firstly, it establishes a useful set of `normal forms'
 $\varphi(\K{X}^1) \subseteq \K{T}+\sp{V}$ under this
 generating set, where the generators from $V$ appear only in $K$-linear combinations. And secondly by~\eqref{eq:psiT},
 the remaining generators $T$ all actually represent elements of the right ideal $R$ for which we are trying to find
 generators and defining relations.

To move from this generating set for $\K{X}^1$ towards the desired generating set for $R$,
note that $\sp{V}$ is a finitely generated (free) $K$-module,
and that it is therefore Noetherian because $K$ is Noetherian.
It follows that the submodule $\psi^{-1}(R)\cap \sp{V}$ is also finitely generated.
Hence there exists a finite set $U$ and a mapping $\pi\colon U\rightarrow \sp{V}$, such that
 $\sp{\pi(U)}=\psi^{-1}(R)\cap \sp{V}$.
Since $\K{T}\subseteq \psi^{-1}(R)$, 
 the modular law implies
\begin{equation}
\label{eq:U}
\psi^{-1}(R) \cap \bigl(\K{T} + \sp{V}\bigr) = \K{T} + \sp{\pi(U)}.
\end{equation}
 As the intersection of $\K{X}^1$-modules, $\K{T} + \sp{\pi(U)}$ is a $\K{X}^1$-module as well.

 Let $Y:=T\cup U$, and extend the mapping $\pi$ to a $K$-algebra homomorphism $\K{Y}\rightarrow \K{Z}$ by setting
 $\pi(t):=t$ for all $t\in T$. For
\[ \overline{\psi}:=\psi\pi\colon \K{Y}\rightarrow \K{X}^1, \]
 it follows from \eqref{eq:U} that $\overline{\psi}(\K{Y})=R$.
In other words, $Y$ is (a pre-image of) a finite generating set of $R$, and it is over this generating set that we will write down a finite presentation for $R$ in the next subsection.

\subsection{Defining relations} \label{sec-rels}

 We seek to find a finite set of defining relations for the right ideal $R$ with respect to the generating set $Y$
 and epimorphism $\overline{\psi}$; this amounts to showing that $I:=\ker\overline{\psi}$ is a finitely generated
 ideal of $\K{Y}$.

One might be tempted to actually view $R$ as a homomorphic image of the $K$-subalgebra $\Sub{T\cup\pi(U)}$ of $\K{Z}$ via the homomorphism $\psi$.
The impediment to this approach is that this subalgebra need not be free, and it is not clear whether it has to be finitely presented. In fact, even the $K$-submodule $\sp{\pi(U)}$ need not be free on the face of it. However, because of the assumption on $K$ being Noetherian,
it has to be finitely presented, and this will play a crucial role in what follows.

To begin, we lift the $\K{X}^1$-module structure from $\K{T}+\sp{\pi(U)}$ to the free module $\K{T}+\sp{U} \subseteq \K{Y}$.
 Let 
\[ \rho\colon \K{T}+\sp{\pi(U)} \to \K{T}+\sp{U} \]
be any map (not necessarily a homomorphism) such that
\begin{equation}
\label{eq:pirho}
\pi\rho = \id_{\K{T}+\sp{\pi(U)}}.
\end{equation}
Since $\pi$ acts as the identity mapping on $\K{T}$, so must $\rho$ as well.
  
\begin{lemma} \label{le:FX1module2}
 The free $K$-module $\K{T} + \sp{U}$ is a right $\K{X}^1$-module via
\[ r\star x := \rho(\pi(r)*x) \text{ for all } r \in T^+\cup U,\ x\in X. \]
\end{lemma}

\begin{proof}
 As for Lemma~\ref{le:FX1module}, the action $\star$ extends to a well-defined action of $\K{X}^1$ from the $K$-basis
 $T^+\cup U$ to all of $\K{T} + \sp{U}$.
\end{proof}

 By definition $*$ and $\star$ are equal on $\K{T}$.
 We next show that the restriction of $\pi$ to  $\K{T} + \sp{U}$ is a $\K{X}^1$-module homomorphism onto $\K{T}+\sp{\pi(U)}$.

\begin{lemma} \label{le:pirp}
 For all $r\in\K{T} + \sp{U}$ and $p\in\K{X}^1$, we have
\[ \pi(r\star p) = \pi(r)*p. \]
\end{lemma}

\begin{proof}
 For $r\in T^+\cup U$ and $p=x$ of $\K{X}^1$, this is obtained by applying $\pi$ to the defining equality
 for $\star$ in Lemma~\ref{le:FX1module2} and by~\eqref{eq:pirho}. 
 Since $\pi$ is $K$-linear, the statement follows.
\end{proof}

 Now, while $\rho$ need not be a $K$-module homomorphism, 
 the following does hold for
\[ M := \ker\pi\cap \sp{U}. \]

\begin{lemma}
\label{la:rhoprops}
 For all $r,s\in \K{T}+\sp{\pi(U)}$, $p\in\K{X}^1$, we have
\[ \rho(r+s)-\rho(r)-\rho(s),\ \rho(r\ast p)-\rho(r)\star p \in M; \]
 i.e., modulo $M$, the mapping $\rho$ is a  $\K{X}^1$-module homomorphism. 
 \end{lemma}

\begin{proof}
 Applying $\pi$ to both expression on the left-hand side yields $0$ by~\eqref{eq:pirho} and Lemma~\ref{le:pirp}.
 Hence both are contained in $\ker\pi\cap\bigl(\K{T}+\sp{U}\bigr) = M$.
\end{proof}

 Next we move on to the description of a set of normal forms for the elements of $R$ in $\K{Y}$ via
\[ \overline{\varphi} := \rho\varphi|_R \colon  R \to \K{T}+\sp{U}. \]
For $p\in R$ we regard $\overline{\varphi}(p)$ as our chosen normal form for it.
At this point we have defined all the algebras, modules and homomorphisms we need for the proof, and we refer the reader once again to Figure \ref{fig:1} for a quick reference about their mutual relationships.

The mapping $\varphi$ is a $\K{X}^1$-module homomorphism by Lemma \ref{le:FX1module}.
Due to the non-linear nature of $\rho$ 
we cannot say the same about $\overline{\varphi}$, 
but, by Lemma~\ref{la:rhoprops}, we have:

\begin{lemma} \label{le:Phipq}
 For all $p,q\in R, m\in\K{X}^1$, we have
\[ \overline{\varphi}(p+q)-\overline{\varphi}(p)-\overline{\varphi}(q),\ \overline{\varphi}(pm)-\overline{\varphi}(p)\star m \in M; \]
i.e., modulo $M$, the mapping $\overline{\varphi}$ is a $\K{X}^1$-module homomorphism.
\end{lemma}

 While $\overline{\psi}$ is a left inverse for $\overline{\varphi}$,
 composing the other way round yields an alternative description for $I=\ker{\overline{\psi}}$:

\begin{lemma} \label{le:kerPsi}
 $\overline{\psi}\overline{\varphi} = \id_R$ and $I = \{r-\overline{\varphi}\overline{\psi}(r)\setsuchthat r\in\K{Y} \}$.
\end{lemma}

\begin{proof}
 The first assertion follows from~\eqref{eq:pirho} and Lemma \ref{le:Psirp}.
 For the second, let $r\in\K{Y}$ to obtain $\overline{\psi}\overline{\varphi}\overline{\psi}(r)=\overline{\psi}(r)$.
 Hence $r-\overline{\varphi}\overline{\psi}(r) \in I$.
 Conversely, if $r\in I$, then $r=r-\overline{\varphi}\overline{\psi}(r)$ is also in the set on the right-hand side.
\end{proof}

This description finally leads us to the desired finite generating set for $I$ as follows.
 Since $K$ is Noetherian, the $K$-module $\sp{\pi(U)}$ is finitely presented. Hence there exists a finite set
 $W_U\subseteq \sp{U}$ such that
\[ \sp{W_U} = M. \]
Define two further finite subsets of $\K{Y}$ as follows:
\begin{align*}
W_Y &:= \bigl\{ y-\overline{\varphi}\overline{\psi}(y)\setsuchthat y\in Y\bigr\},\\
W_{Y,Y} &:=\bigl\{zy-z\star \overline{\psi}(y) \setsuchthat y,z\in Y\bigr\}. 
\end{align*}

\begin{lemma} \label{le:finrel}
 $I$ is generated by $W := W_U\cup W_Y\cup W_{Y,Y}$ as an ideal of $\K{Y}$.
\end{lemma}

\begin{proof}
 Let $J := \Id{W}$.
 For the containment $J\subseteq I$, first recall that $W_U\subseteq\ker\pi\subseteq \ker\overline{\psi}=I$.
 Next $W_Y \subseteq I$ by Lemma~\ref{le:kerPsi}.
 Finally $W_{Y,Y} \subseteq I$ follows by applying $\overline{\psi}$ to its elements and
 using that $\overline{\psi}$ is both a $K$-algebra homomorphism and a $\K{X}^1$-module homomorphism by
 Lemmas~\ref{le:Psirp} and~\ref{le:pirp}.
 
 For the reverse containment $I\subseteq J$, by Lemma \ref{le:kerPsi}, it suffices to show that $J$ contains
 all polynomials $r-\overline{\varphi}\overline{\psi}(r)$ with $r\in\K{Y}$.
 We prove this claim by induction on the complexity of $r$. For $r=y$ in $Y$ we have a generator from $W_Y$.

 Suppose now that $r \equiv \overline{\varphi}\overline{\psi}(r) \bmod J$ for a monomial $r\in Y^*$.
 By the definition of $\overline{\varphi}$ we have $r_t\in\K{T}^1$,  $\gamma_u\in K$ such that
\[ \overline{\varphi}\overline{\psi}(r) = \sum_{t\in T} r_t t + \sum_{u\in U} \gamma_u u. \]
 For $y\in Y$ we obtain
\begin{align*}
  & \overline{\varphi}\overline{\psi}(ry) \\
  = &  \overline{\varphi}\bigl(\overline{\psi}(r)\overline{\psi}(y)\bigr) \\
 \equiv & \overline{\varphi}\overline{\psi}(r)\star \overline{\psi}(y) \mod J & \text{using generators $W_U$ and Lemma~\ref{le:Phipq}}\\
 = & \Bigl(\sum_{t\in T} r_t t + \sum_{u\in U} \gamma_u u\Bigr)\star\overline{\psi}(y) \\
 = & \sum_{t\in T} r_t (t\star\overline{\psi}(y))  + \sum_{u\in U} \gamma_u  (u\star\overline{\psi}(y)) & \text{since } \star \text{ and } * \text{ are equal on } \K{T} \\
 \equiv& \sum_{t\in T} r_t ty  + \sum_{u\in U} \gamma_u uy \mod J & \text{using generators $W_{Y,Y}$}  \\
 \equiv& ry \mod J & \text{by the induction assumption.}
\end{align*}
 Hence $r\equiv\overline{\varphi}\overline{\psi}(r)\bmod J$ for all monomials $r\in Y^*$.
 Since $\overline{\psi}$ is $K$-linear and $\overline{\varphi}$ is $K$-linear modulo $\sp{W}$ by Lemma~\ref{le:Phipq},
 the result follows for all polynomials $r\in\K{Y}$. 
\end{proof}

 Since $I$ is finitely generated by Lemma~\ref{le:finrel},
\[ \K{Y} / I \cong R \]
 is a finitely presented $K$-algebra. This concludes the proof of Lemma \ref{le-o3}.

\section{Finite generation (Theorem~\ref{cor-i5})}
\label{sec-conca}

 This section constitutes a self-contained proof of Theorem~\ref{cor-i5}:
{\it Let $A$ be a $K$-algebra (not necessarily commutative or with $1$) over a commutative ring
 $K$ with $1$, and let $B$ be a subalgebra of $A$ such that $A/B$ is a finitely presented $K$-module.
Then $A$ is finitely generated as a $K$-algebra  if and only if $B$ is finitely generated as a $K$-algebra.}

The `if' statement is easy: if $X$ generates $B$ as a $K$-algebra, and $\{y+B\setsuchthat y\in Y\}$ generates the $K$-module $A/B$, then
the set $X\cup Y$ generates $A$ as a $K$-algebra.

For the `only if' direction, begin by noting that
finite $K$-module generation of $A/B$ means that there exists a finite set $Y$ and a mapping $Y\rightarrow A$, $y\mapsto \overline{y}$, such that $A=B+\sp{\overline{Y}}$.
Since $A$ is finitely generated as a $K$-algebra, there exists a finite set $X\supseteq Y$, and a mapping $X\rightarrow A$, $x\mapsto \overline{x}$, which extends $y\mapsto \overline{y}$, and such that the induced $K$-algebra homomorphism $\K{X}\rightarrow A$, $p\mapsto \overline{p}$, is surjective.
Next we note that $\sp{\overline{Y}}/(B\cap\sp{\overline{Y}})\cong A/B$,
and then finite presentability of $A/B$ as a $K$-module implies that $B\cap\sp{\overline{Y}}$ is a finitely generated $K$-module.
Thus, there exists a finite set $Z\subseteq \sp{Y}\subseteq \K{X}$ such that $\sp{\overline{Z}}=B\cap\sp{\overline{Y}}$.

From $A=B+\sp{\overline{Y}}$ it follows that there is a mapping $\gamma\setsuchthat \K{X}\rightarrow\sp{Y}$
such that
$\overline{p-\gamma(p)}\in B$ for all $p\in \K{X}$.
Write $X^k$ ($k\in\N$) for the set of all monomials of length $k$ over $X$, and let
\[
U:=\{ p-\gamma(p)\setsuchthat p\in X\cup X^2\cup X^3\}.
\]
Finally let $V:=U\cup Z$, which is clearly a finite set. We claim that $\Sub{\overline{V}}=B$.

From the foregoing definitions we immediately have $\overline{V}\subseteq B$. So we only need to show that $B\subseteq \Sub{\overline{V}}$. In other words, we need to show
that for every $p\in \K{X}$ with $\overline{p}\in B$ we have $\overline{p}\in \Sub{\overline{V}}$.
We do this by induction on the degree $m$ of $p$.

If $m\leq 3$, write 
\[ p=\sum_{r\in X\cup X^2\cup X^3} \alpha_r r. \]
Then
\[
p=\sum_{r\in X\cup X^2\cup X^3} \alpha_rr= \sum_{r\in X\cup X^2\cup X^3} \alpha_r (r-\gamma(r))+\sum_{r\in X\cup X^2\cup X^3}  \alpha_r\gamma(r).
\]
Clearly
\[
\sum_{r\in X\cup X^2\cup X^3}  \alpha_r(r-\gamma(r))\in \Sub{U}.
\]
Further, since $\overline{p}$ and all $\overline{r-\gamma(r)}$ belong to $B$, it follows that
\[
\overline{\sum_{r\in X\cup X^2\cup X^3} \alpha_r\gamma(r)} \in B\cap \sp{\overline{Y}}=\sp{\overline{Z}}.
\]
Therefore
\[
\overline{p}\in \Sub{\overline{U}}+\sp{\overline{Z}}\subseteq \Sub{\overline{V}},
\]
as required.

Now consider the case where $m>3$, and assume inductively that the assertion holds for all polynomials of degree less than $m$.
Write
\[
p=\sum_{r\in X^m} \alpha_r r+ p^\prime,
\]
where the degree of $p^\prime$ is smaller than $m$.
Write each monomial $r\in X^m$ as $r=r_1r_2\cdots r_l$,
where each of $r_1,\dots, r_{l-1}$ has length $2$, and $r_l$ has length $2$ or $3$ 
(depending on the parity of $m$).
Recall that $Y\subseteq X$, so $\gamma(r_1),\dots,\gamma(r_l)$ are linear polynomials over $X$.
It follows that
\[
r=(r_1-\gamma(r_1))(r_2-\gamma(r_2))\cdots (r_l-\gamma(r_l))+r^\prime,
\]
where $r^\prime$ is a polynomial of degree less than $m$.
Now we have
\[
p=\sum_{r\in X^m} \alpha_r(r_1-\gamma(r_1))\cdots (r_l-\gamma(r_l))+\sum_{r\in X^m} \alpha_r r^\prime+ p^\prime.
\]
Clearly,
\[
\sum_{r\in X^m} \alpha_r(r_1-\gamma(r_1))\cdots (r_l-\gamma(r_l))\in\Sub{U},
\]
and 
\[
\overline{\sum_{r\in X^m} \alpha_r r^\prime +p^\prime}\in B.
\]
Since the degree of
\[
\sum_{r\in X^m} \alpha_r r^\prime + p^\prime
\]
is less than $m$, by induction we have
\[
\overline{\sum_{r\in X^m} \alpha_r r^\prime +p^\prime}\in \Sub{\overline{V}},
\]
and hence 
\[
\overline{p}\in \Sub{\overline{U}}+\Sub{\overline{V}}\subseteq \Sub{\overline{V}},
\]
completing the induction and the proof of Theorem~\ref{cor-i5}.

\section{A comparison with groups, semigroups and Lie algebras}
\label{sec-concb}

 Throughout this paper we considered algebras that do not necessarily have an identity.
 Note that presentations are defined with respect to free algebras, and there is a difference between
 the free $K$-algebra $\K{X}$ and the free $K$-algebra with $1$, which is $\K{X}^1$, over the same set $X$.
 However, a $K$-algebra $A$ that contains a multiplicative identity is finitely presented as an algebra (without $1$)
 if and only if it is finitely presented as an algebra with $1$. That is, $A\cong\K{X}/\Id{U}$ for some finite $X$
 and finite $U\subseteq\K{X}$ iff $A\cong\K{Y}^1/\Id{V}$ for some finite $Y$ and finite $V\subseteq\K{Y}^1$.
 We omit the straightforward but slightly technical argument. Consequently all our results from Section~\ref{sec-intro}
 hold in the formulation for $K$-algebras (rings) with $1$ as well.

 The condition that $A/B$ is a finitely presented $K$-module is essential for our finite generation result
 (Theorem~\ref{cor-i5}) and a fortiori for our finite presentation result (Theorem~\ref{thm-i1}) as the following example shows:
\begin{exa}
\label{exa-1}
Let $K:=\Z[x_1,x_2,\dots]$, the integral polynomial ring over infinitely many commuting variables.
Consider $K$ as a $K$-module, and turn it into a $K$-algebra $A$
with the trivial multiplication $ab=0$ for all $a,b\in A$.
Clearly, $A\cong \K{y}/\Id{y^2}$ is finitely presented and, in particular, finitely generated.
Let $B$ be the subalgebra of $A$ generated by $\{x_1,x_2,\dots\}$.
Then the quotient $A/B$ is a cyclic $K$-module, but not finitely presented.
Moreover $B$ itself is not finitely generated as a $K$-algebra, let alone finitely presented.
\end{exa}
 One can ask whether perhaps in Theorem~\ref{thm-i1} the Noetherian property of $K$ may be replaced by the weaker
 requirement that $A/B$ is a finitely presented $K$-module -- just like in Theorem~\ref{cor-i5}. This, however, 
 is not the case, as demonstrated by the next example:
\begin{exa}
\label{exa-2}
Let 
\[K:=\Z[x_1,x_2,\dots]/\Id{x_ix_j\setsuchthat i,j=1,2,\dots}.\]
 Considered as a $K$-algebra, $A := K$ has a multiplicative identity and consequently its multiplication is not zero.
 Moreover, $A$ is finitely presented, since  $A\cong \K{y}/\Id{y^2-y}$. 
Its ideal $B := \Id{x_1}$ consists of all integer multiples of $x_1$; that is, 
$B$ is finitely generated as a $K$-module, and also as a $K$-algebra.
It immediately follows that $A/B$ is a finitely presented $K$-module.
However, $B \cong \K{z}/\Id{z^2,x_1z, x_2z,\dots}$ is not finitely presented as a $K$-algebra, because no ideal generator $x_i z$ can be expressed in terms of the other generators.
\end{exa}

 Comparing our results to those of Lewin~\cite{lewin69} and Voden~\cite{voden09} for algebras over fields,
 one essential difference is that submodules of free modules over rings are not necessarily free.
 In contrast, subspaces of vector spaces are free, which is crucial for Lewin's proof that subalgebras of finite
 co-dimension in finitely generated free algebras over fields are finitely presented~\cite[Section VII]{lewin69}.
 For $K$ a field, our proof of Lemma~\ref{le-o3} would simplify in that $\pi\colon U\to\sp{V}$ as defined at the
 end of Section~\ref{sec-gens} is essentially the identity map, and the desired relations can be found using
 $\psi,\varphi$ within $\K{T}+\sp{V}$ without having to pass to $\overline{\psi},\overline{\varphi}$ in
 Section~\ref{sec-rels}. 

Our results can be viewed as direct analogues of those of
Reidemeister~\cite{reidemeister27} and Schreier~\cite{schreier27} from combinatorial group theory,
asserting that a subgroup of finite index in a finitely generated (resp. finitely presented) group is itself finitely generated (resp. finitely presented).
These results have over the years played a central role in the development of combinatorial group theory; see, for example,
\cite[Section II.4]{lyndon01}, and for a more historical overview 
\cite[Section II.3]{chandler82}.
It might be therefore interesting to compare 
these theorems with the results obtained here.

For example, the finite co-rank property specialised to the case of a ring $R$ and a subring $S$ requires the quotient
$R/S$ to be a finitely generated abelian group.
This appears weaker than the finite index condition in groups, where the quotient $G/H$ is merely required to be finite. One may wonder whether
 the latter condition can be weakened, at least in the case of a normal subgroup $N$, to requiring the quotient $G/N$ to be finitely generated abelian.
This, however, is known not to be sufficient to ensure preservation of finite generation or presentability.
Indeed, for the free group $F_2$ of rank $2$, its derived subgroup $F_2^\prime$ is not finitely generated as a subgroup.

Probably the strongest methodological parallel between the Reidemeister--Schreier Theorem and our result is provided by the role of the mapping $\varphi$ in Section
\ref{sec-ri}, which was used to `rewrite' 
the elements of $\K{X}$ representing the elements of the right ideal $R$ into suitable 
polynomials over $Y$.
In Magnus--Karrass--Solitar's rendering of Reidemeister--Schreier  in
\cite[Section 2.3]{mks} the analogous role is played by what they call a \emph{rewriting process} $\tau$.
A good explicit choice of these rewriting mappings is then crucial for both the group- and ring-theoretic proofs.

To continue the comparison, note that the group-theoretic proof \emph{could}
 be couched along the same lines as our proof above,  namely proving that if $N$ is a normal subgroup of finite index in a free group $F_n$ of finite rank $n$ then
\begin{enumerate}[label=(\arabic*),leftmargin=7mm]
\item[(1)] $N$ is finitely presented; and
\item[(2)]
any normal subgroup $M$ of $F_n$ which is contained in $N$ and which is finitely generated as a normal subgroup of $F_n$
is also finitely generated as a normal subgroup of $N$.
\end{enumerate}
Now, in the group case, (1) follows from the stronger result that subgroups of free groups are free (the Nielsen--Schreier Theorem).
Admittedly, sometimes this is proved as a consequence of the Reidemeister--Schreier Theorem (as, for example, in \cite{mks}), but more transparent proofs, relying only on alternative descriptions of the free group are available. For instance, it is known that a group is free if and only if it acts freely on a tree \cite[Theorem 4]{serre}.
But then it is obvious that every subgroup also acts freely on the same tree, and hence it must be free
\cite[Theorem 5]{serre}.
Statement (2) is also well known and straightforward to prove: 
if $M$ is generated as a normal subgroup of $F_n$ by a set $Y$, and if $C$ is a set of coset representatives of
$N$ in $F_n$, then $M$ is generated by $\{ c^{-1}yc\::\: y\in Y,\ c\in C\}$ as a normal subgroup of $N$.

In our proof for rings it is Lemma \ref{le-o3} corresponding to (1) that presents the greatest
 difficulties, and we have not been able to find a proof that would avoid its technical nature. 
By way of contrast, Lemma \ref{le-o4} is comparable in spirit and difficulty to (2). 
This is perhaps related to the fact that subrings of free rings are not free in general, and the finitely generated ones may even fail to be finitely presented.
 We give an explicit example to demonstrate this.

\begin{exa}
Let $S$ be the subsemigroup of $\{a,b,c\}^+$ generated by $v:=ba, w:=ba^2,x:=a^3,y:=a^2c,z:=ac$.
In~\cite[Example 4.5]{crrt} it is proved that $S$ is not finitely presented.
From the arguments given there, it follows easily that $S$ is defined by the infinite presentation
$\langle v,w,x,y,z \ |\ vx^ny=wx^nz \text{ for } n\in\N_0\rangle$.

The subring $R := \Sub{v,w,x,y,z}$ of the free ring $\ZZ{a,b,c}$ is precisely the integral semigroup ring of $S$.
 From the semigroup presentation of $S$, it is straightforward that $R\cong \ZZ{v,w,x,y,z}/I$ for the ideal
 $I$ generated by 
\[ \{ vx^ny-wx^nz \setsuchthat n\in\N_0 \}. \]
 Seeking a contradiction, suppose that $R$ is finitely presented. Then there exists $N\in\N_0$ such that
 $W := \{ vx^ny-wx^nz \setsuchthat 0\leq n\leq N\}$ also generates $I$.
 In particular, $vx^{N+1}y-wx^{N+1} z \in \Id{W}$, i.e.
\[
vx^{N+1}y-wx^{N+1} z= \sum_{i=0}^{N} p_i (vx^iy-wx^i z)q_i
\]
 for some $p_i,q_i\in\ZZ{v,w,x,y,z}^1$.
 However, none of $vx^iy$, $wx^i z$ ($i\leq N$) is a factor of $vx^{N+1}y$.
 So the monomial $vx^{N+1}y$ does not appear on the right-hand side, a contradiction. Thus $R$ is not finitely presented.
\end{exa}

 Further, it turns out that ideals and $K$-subalgebras of finite co-rank in free $K$-algebras are \emph{never}
 free when $K$ is an integral domain.

\begin{lemma}
\label{le:Iunfr}
 If $K$ is an integral domain (not necessarily Noetherian), then no proper non-trivial ideal $I$ of a free
 $K$-algebra $\K{X}$ is free as a $K$-algebra.
\end{lemma}

\begin{proof}
 Suppose that $I$ is a free $K$-algebra, say with $\K{Y}\rightarrow I$, $t\mapsto \overline{t}$, an isomorphism.
 Let $x\in X$, $y\in Y$. 
Since $I$ is an ideal, $x\overline{y},\overline{y}x\in I$. 
Hence we have $p,q\in \K{Y}$ such that
 $x\overline{y}=\overline{p}$, $\overline{y}x=\overline{q}$. 
Moreover $\overline{y}x\overline{y} = \overline{yp} = \overline{qy}$. 
Since $t\mapsto \overline{t}$ is an isomorphism, this yields $yp=qy$ in $\K{Y}$, and hence $p = ry$ 
 for some $r\in \K{Y}$. Consequently 
 $x\overline{y} = \overline{ry}$. 
Since $K$ has no zero-divisors, neither does $\K{X}$. Hence $x=\overline{r}\in I$,
 implying $I=\K{X}$, a contradiction.
\end{proof}

 Similarly, subalgebras of finite rank are generally not free.

\begin{lemma}
\label{le:Sunfr}
 If $K$ is a Noetherian integral domain, no proper $K$-subalgebra $S$ of finite co-rank in a free $K$-algebra $\K{X}$ is free.
\end{lemma}

\begin{proof}
Suppose that $S$ is free, say $S\cong \K{Y}$, and let $r\mapsto \overline{r}$ be an isomorphism $\K{Y}\rightarrow S$.
Let $I$ be an ideal of finite co-rank in $\K{X}$ contained in $S$;
it exists by Lemma~\ref{le-o1}.

Note that cyclic $K$-subalgebras of free $K$-algebras never have finite co-rank, with the trivial exception
of $\K{x}$ as a $K$-subalgebra of itself. Hence we may asume that $|Y|>1$, and let
$y,z\in Y$ with $y\neq z$.
Consider the collection of monomials $s_i := y^2z^iyz$ for $i\in\N$.
 Since $K$ is Noetherian and $I$ has finite co-rank, 
 the $K$-submodule $\sp{s_1+I,s_2+I,\dots}$ of $S/I$ is finitely generated.
Hence
there exist $\gamma_i\in K$, such that all but finitely many equal to $0$, 
one of them, say $\gamma_n$, equals to $1$, and
 $t := \sum_{i\in\N} \gamma_i s_i$ such that $\overline{t}\in I$.

Now consider an arbitrary $x\in X$. Since $I$ is an ideal, we have
$\overline{t}x,x\overline{t}\in I$.
Hence there exist $p,q\in \K{Y}$ such that $\overline{t}x=\overline{p}$ and $x\overline{t}=\overline{q}$.
This implies
\begin{equation}
\label{eq-conc1}
\overline{pt}=\overline{t}x\overline{t}=\overline{tq}.
\end{equation}
Since $r\mapsto \overline{r}$ is an isomorphism, it follows that $pt=tq$ in $\K{Y}$.
 From the non-overlapping properties of $s_i$, we see that every monomial in $p$ must have some $s_i$ as prefix, and,
 likewise, every monomial in $q$ must have some $s_i$ as suffix.
Writing $p := \sum_{i\in\N}s_ip_i$ and $q := \sum_{j\in\N} q_js_j$ with $p_i,q_j\in \K{Y}^1$, from \eqref{eq-conc1} we obtain
\[
 \sum_{i,j} \gamma_j s_ip_is_j = \sum_{i,j} \gamma_i s_iq_js_j.
\]
 Again, the non-overlapping properties of the $s_i$ imply that $\gamma_j s_ip_is_j = \gamma_i s_iq_js_j$,
 and hence $\gamma_j p_i = \gamma_iq_j$, for all $i,j\in\N$. 
 Since $\gamma_n = 1$, we obtain $p_n=q_n$, and then $p_i = \gamma_i q_n=\gamma_i p_n$ for all $i\in\N$.
 Then $p = tp_n$.
But now, in $\K{X}$ we have
\[
\overline{t}x=\overline{p}=\overline{t}\,\overline{p}_n,
\]
and hence $x=\overline{p}_n$ because $K$ has no zero-divisors.
This proves that all $x\in X$ belong to $I$, and so $S=I=\K{X}$, a contradiction.
\end{proof}

Similar discussions of (non-)freeness for subalgebras of free algebras over a field
are presented in
\cite[Section 6.6]{cohn85}.

 Even though the subsemigroups of a free semigroup need not be finitely presented, at least there is an algorithmic
 procedure which decides whether any given finitely generated subsemigroup of $X^\ast$ is finitely presented or not;
 see \cite[Section 5.2]{lallement79}. It would be interesting to know whether there is an analogue for free rings:

\begin{que}
Is there an algorithm which decides for any finite subset $T$ of a free ring 
$\ZZ{X}$ whether the subring $\Sub{T}$ is finitely presented?
\end{que}

Of course, one can consider this question more generally for $K$-algebras, but presumably the answer would depend partly on the nature of $K$.

In free semigroups it is also the case 
that ideals and one-sided ideals of $X^+$ that are finitely generated as semigroups are also finitely presented (although not necessarily free); see \cite[Corollary 3.6, Theorem 4.3]{crrt}. This leads us to pose the following:

\begin{que}
Are the (one- or two-sided) ideals of a free ring that are finitely generated as rings necessarily finitely presented?
\end{que}

Finally, one might wonder whether the results of this paper can be extended to the context of non-associative algebras. The answer, however, is negative already for Lie algebras, as the following example shows. Specifically, we exhibit an ideal of co-dimension $1$ in a free Lie algebra of rank 2 which is not finitely generated as a Lie algebra.

\begin{exa}
\label{exa-lie}

Let $L(x,y)$ be the free Lie algebra on $\{x,y\}$ over a field $K$.
Note that the free Lie algebra over $\{x\}$ is isomorphic to $K$ with trivial Lie bracket. Define a homomorphism
$h$ by
\[
h\colon L(x,y)\rightarrow K,\ x\mapsto 1,\ y\mapsto 0.
\]
 Clearly $B := \ker h$ has co-dimension $1$ in $L(x,y)$. We will show that $B$ is not finitely generated
 as a Lie algebra.

First recall that the free $K$-algebra $\K{x,y}$ is the enveloping algebra of $L(x,y)$. We will identify elements
 in $L(x,y)$ with their images in $\K{x,y}$ using $[a,b] = ab-ba$. As a subalgebra of a free Lie algebra,
 $B$ is free over some set $Y$ by the Shirshov--Witt Theorem. Moreover $|Y|$ is equal to the rank of the right
 ideal $I := B \K{x,y}$ as a right $\K{x,y}$-module~\cite[Section 2.5.1]{Re:FLA}.

 We claim that 
\begin{equation}\label{eq:Ispan}
 I = \{x^n y \setsuchthat n\in\N_0 \}\ \K{x,y}.
\end{equation}
 For $\subseteq$ note that each non-trivial monomial $b\in B$, when written using Lie-brackets, must contain $y$. 
 Hence the representation of $b$ in $\K{x,y}$ is a sum of monomials  all of which
 contain $y$.

 For the converse inclusion $\supseteq$ in~\eqref{eq:Ispan} we show that $x^ny\in I$ by a straightforward
 induction on $n\in\N_0$. 
For the base case note that $y\in B\subseteq I$ by the definition of $h$. Since
\[
 B\ni [\dots [[y,\underbrace{x],x],\dots,x}_n]=x^ny+\sum_{i=1}^n (-1)^{n-i} \binom{n}{i} x^{n-i} y x^{i},
\]
 the induction assumption yields $x^ny\in I$. Thus~\eqref{eq:Ispan} is proved.

Now it is easy to see that $\{x^n y \setsuchthat n\in\N_0 \}$ is a free basis for the
$\K{x,y}$-module $I$.
Consequently $B$ is a free Lie algebra of infinite rank and is
 not finitely generated.
\end{exa}

\textbf{Acknowledgement.}
The authors are thankful to an anonymous referee of an earlier version of the paper
for their suggestions how to improve the exposition in the paper.
We are particularly grateful for the referee's enhancement of Example \ref{exa-1}, the substance of Example \ref{exa-2}, and for asking a question which led us to Theorem \ref{cor-i5} in its present form.

\def\cprime{$'$}

\end{document}